\documentclass{amsart}
\usepackage{graphicx} % Required for inserting images

\title{On Nearly Frobenius structures}
\author{William Davis and Olivia Dumitrescu}
\address{University of North Carolina at Chapel Hill, 340 Phillips Hall, CB 3250 NC 27599-3250 email:dolivia@unc.edu, willdavismath@gmail.com}

\date{October 2025}

\newcommand{\bea}{\begin{eqnarray}}
\newcommand{\eea}{\end{eqnarray}}
\newcommand{\be}{\begin{equation}}
\newcommand{\ee}{\end{equation}}
\newcommand{\tensor}{\otimes}
\newcommand{\Hom}{\mathrm{Hom}}
\newcommand{\Atn}{A^{\tensor n}}
\newcommand{\Atm}{A^{\tensor m}}

\newcommand{\inv}{^{-1}}
\newcommand{\eps}{\varepsilon}

\newcommand{\id}{\mathrm{id}}
\newcommand{\eul}{\mathbf{e}}
\newcommand{\bige}{\mathbf{E}}
\newcommand{\real}{\mathbb{R}}
\newcommand{\cpx}{\mathbb{C}}

\newcommand{\cob}{$\mathbf{2Cob}$\,}

\usepackage{amsmath}
\usepackage{amsthm}
\usepackage{amsfonts}
\usepackage{bbm}
\usepackage{amssymb}
\usepackage{geometry}
\usepackage{tikz}
\usepackage{tikz-cd}
\usepackage{txfonts}
\usepackage[inline, shortlabels]{enumitem}
\usepackage{blkarray}
\usetikzlibrary{arrows,patterns,decorations}
\usepackage{titleref}
\usepackage[normalem]{ulem} 
\usepackage{float}
\usepackage{graphicx}% Include figure files
\usepackage{dcolumn}% Align table columns on decimal point
\usepackage{mathptmx}
\usepackage{etoolbox}
\usepackage{xcolor}%subcaption}
\usepackage[caption=false]{subfig}
\usepackage{tabularx}
\usepackage{hyperref}

\newtheorem{Lem}{Lemma}[section]
\newtheorem{Thm}[Lem]{Theorem}
\newtheorem{Cor}[Lem]{Corollary}

\newtheorem{Prop}[Lem]{Proposition}

\theoremstyle{definition}
\newtheorem{Ex}[Lem]{Example}
\newtheorem{Def}[Lem]{Definition}
\newtheorem{Rmk}[Lem]{Remark}

\usepackage[ maxbibnames=10, backend=biber]{biblatex}
\defbibheading{References}{References}

\begin{document}

\begin{abstract}
Nearly Frobenius structures and 2-dimensional Almost TQFTs were introduced and shown to be in categorical equivalence in \cite{GLSU} in the attempt to extend the Atiyah-Segal's definition to the category of infinite dimensional vector spaces.
In this paper, we investigate nearly Frobenius structures and we give a classification result for Almost TQFTs in dimension 2. In particular, the TQFTs functorial axioms become equivalent to the Edge Contraction/Construction axioms of colored ribbon graphs recently introduced by the first author in \cite{thesis}. 
\end{abstract}


\maketitle

\tableofcontents

%\maketitle

\section{Introduction}

Frobenius algebras and two-dimensional topological quantum field theories (2D TQFT) have long been of interest across a spectrum of fields in mathematical physics. First introduced by  Atiyah \cite{Atiyah} and largely classified by \cite{MS}  Moore and Segal, 2D TQFTs serve as a bridge, weaving connections from string theory to enumerative geometry, from topological recursion to categorification. Two dimensional TQFTs are in one to one correspondence with Frobenius algebras as noted by Dijkgraaf in his PHD thesis \cite{Dijkgraaf}. Important to the structure of commutative Frobenius algebras is that they are by necessity finite-dimensional (\cite{Kock}). Notable examples of such algebras include group algebras of finite groups or cohomology rings of compact manifolds. While  the structure of  similar infinite-dimensional  analogs is less understood,  a recent work of Gonzalez and al. \cite{GLSU}  introduces  infinite-dimensional {\it nearly Frobenius algebras} and their associated {\it Almost TQFT} proving they are in categorical equivalence. While the authors of \cite{GLSU}  develop their theory from a largely algebraic perspective, we reframe these ideas into the setting of {\it enumerative geometry} and graph theory compatible with the topological recursion formalism.

 The quantization can be regarded as a reconstruction process from the two-dimensional topological quantum field theories  of Atiyah-Segal to cohomological field theories  of Kontsevich-Manin \cite{KM}.
 The 2D TQFT classification for Frobenius algebras,  was indirectly used in the celebrated classification of Givental-Teleman for Cohomological Field Theories for semisimple Frobenius algebras \cite{Teleman_2011}.  
 In \cite{DM3} (in progress) the authors reformulated an enumerative version of the Teleman's classification theorem for Cohomological Field Theories \cite{Teleman_2011} in the enumerative geometry formalism via the theory of ribbon graphs using the works of Thurston et al., Strebel, Mulase, and Penkava \cite{STT}, \cite{Strebel}, \cite{MP}. A recent work of Kontsevich et al. \cite{KTV} also uses ribbon graphs to describe moduli spaces related to fully-extended oriented 2D TQFT for non-compact 2-cobordisms given by a construction of Lurie \cite{Lurie}.

\noindent In this paper, in Sections \ref{Section2} and \ref{Frob} we introduce Frobenius algebras and TQFT, from Atiyah's categorical approach. We follow works of Segal \cite{Segal}, Dijkgraaf \cite{Dijkgraaf}\, Moore \cite{MS}, Teleman \cite{Teleman_2011}, and Kock \cite{Kock} exploring the connections bridging topological invariants to algebraic structures. Frobenius algebras were also studied by Nakayama in the early 20th century from a different perspective \cite{Nakayama1}, \cite{Nakayama2}.

   \begin{Thm}\label{thm B}
        Suppose $A$ is a Frobenius algebra with counit $\eps$, coproduct $\delta$, and Euler element $\eul$. Then, the value of a 2D TQFT is given by 
        \be\label{eqn:tqftclassificationintro}
        \omega_{g,n,m}(v_1,\ldots,v_n)=\begin{cases}
            \eps(v_1\cdots v_n\eul^g),& m=0,\\
            v_1\cdots v_n\eul^g,& m=1,\\
            \delta^{m-1}(v_1\cdots v_n\eul^g),& m\geq 2.
        \end{cases}
        \ee
    \end{Thm}

\noindent The case $m=0$ was proved in Theorem 3.8 of \cite{DM2}, see also \cite{DM1}; therefore this statement for $m\geq 1$ can be regarded as a generalization of results of \cite{DM1}.

 We recall that a {\it ribbon graph} is a graph labeling faces embedded on a punctured Riemann surface \cite{Grothendieck}, \cite{Hooft}, while in this work we use {\it cell graph} to denote the Poincar\'e dual of a ribbon graph, namely a ribbon graph itself labeling vertices. For a Frobenius algebra,  the authors define in \cite{DM1, DM2} a {\it cell graph TQFT} via a set of axioms on the category of cell graphs on a fixed genus and a fixed number of marked points. This set of axioms induced by pairs of pants decomposition on punctured Riemann surfaces reflects the given Frobenius algebra structure. The authors further prove that a cell graph TQFT is independent on the cell graph, depending only on its topological information, namely on the associated punctured Riemann surface described by the cell graph. Furthermore, for a given Frobenius algebra, a classification result for cell graphs 2D TQFT and Theorem \ref{thm B} imply that the axioms of Atiyah-Segal and the ones of cell graphs TQFTs are equivalent. This observation was exploited by the authors of \cite{DM1} in enumerative geometry problems. For example, the axioms of cell graph TQFTs are also determining a recursion of generalized Catalan numbers via a count of cell graphs. Even if the generalized Catalan recursion uses only topological information, the generating functions of the generalized Catalan numbers, contains information of Gromov-Witten invariants of a point \cite{Kontsevich_intersection},\cite{W1991}; these results proved in \cite{DMSS}, \cite{DM1}, \cite{DM2} and will be also presented in \cite{DaDu}. The authors of \cite{DM1, DM2} have used and exploited  different generalization of their Catalan recursion in the framework of enumerative geometry, namely Hitchin theory, Hurwitz numbers, or the Cohomological Field Theories.  In this work we will restrict to describing nearly Frobenius structures.

 \noindent In Sections \ref{Section3} and \ref{Section4}, the authors investigate the concepts of Nearly Frobenius algebras and Almost TQFTs, i.e. infinite dimensional vector spaces that can be endowed with a Frobenius like structure. Throughout this work, for a counital Frobenius algebra $A$, we will use $\eps$ to denote the counit, $\delta$ the coproduct, $m$ the product, $\eta=\eps\circ m$ the non-degenerate bilinear form, $\bige=m\circ\delta$ the Euler map \eqref{eqn:EULbasis}, and $\eul=m\circ\delta(1)$ the Euler element, if it exists. The difficulty in the infinite dimensional case is that $A$ is not necessarily unital,  so the Euler element of $A$ may not exist. Moreover, the duality between $A$ and its dual $A^*$ is no longer an isomorphism, so the TQFT can not be recovered from a map with output in the ground field $K$. In the finite unital case, for Frobenius algebras, the classification result was achieved by gluing cobordisms to the input side of the TQFT i.e. by utilizing the coproduct of the unit $\delta(1)$ \cite{DM1, DM2}. In the nearly Frobenius algebra the coproduct of the unit is no longer available, so the classification result requires a new approach. Nevertheless, if $A$ is not unital, one can instead glue certain cobordisms to the output side of a TQFT, as long as $A$ is counital (i.e. the map $\eps\in A^*$ exists). These techniques look similar to the techniques used in the unital case where now the bilinear form $\eta=\eps\circ m$ plays an analogous role to $\delta(1)$ in the previous proof.

\begin{Thm}\label{thm C} 
    Let $A$ be a counital nearly Frobenius algebra with counit $\eps$, coproduct $\delta$, and Euler map $\bige$. The value of the Almost TQFT associated to $A$ is then given by
        \be
     \omega_{g,n,m}(v_1,\ldots,v_n)=\begin{cases}
            \eps(\bige^g(v_1\cdots v_n)),& m=0,\\
            \bige^g(v_1\cdots v_n),& m=1,\\
            \delta^{m-1}(\bige^g(v_1\cdots v_n)),& m\geq 2.
        \end{cases}
    \ee
\end{Thm}

We further recall, that starting from a nearly Frobenius algebras,   the authors define in \cite{thesis}, see also \cite{DaDu}, an equivalent combinatorial set of axioms based on the formalism of ribbon graphs, also known as Grothendieck's dessins d'enfants \cite{Grothendieck}, and classify all such possible assignments (see Theorem \eqref{thm A}). The classification result of \cite{thesis}, \cite{DaDu} restricted to Frobenius algebra, ie to the finite dimensional case, recovers the 2D TQFT via ribbon graphs as exposed in \cite{DM1},  \cite{DM2} or \cite{D}.

\noindent Since the systems of linear maps formulated from the cobordism perspective and the ribbon graph perspective give the same values, Theorems \ref{thm A} and \ref{thm C}  imply that for nearly Frobenius structures, the standard set of axioms of 2D TQFT are equivalent to the ones defined on the category of ribbon graphs see Theorem 1.1, Chapter 4 of \cite{thesis} and Theorem 1.2, Chapter 5 of \cite{thesis} and also \cite{DaDu}. We obtain the following result:

\begin{Cor}\label{thm D}
    The functorial axioms for {\it Almost TQFT} of \cite{GLSU} are equivalent to the {\it colored ribbon graph Edge Contraction/Construction} axioms.
\end{Cor}

The combinatorial information of the ribbon graph formulation of Almost TQFT  is closely related to mathematical physics. Therefore in this work, we will omit the technical details of the ribbon graph formulation and rather emphasize the algebraic aspects of the theory.
We will briefly outline interesting research direction to relate this work with enumerative geometry.
\noindent The ribbon graphs theory provides a recursion of generalized Catalan numbers similar to the Cut and Join equation for Hurwitz numbers. The two recursions are important in mathematical physics as they compute Gromov-Witten invariants of a point \cite{Kontsevich_intersection}, \cite{W1991}, \cite{DMSS} via the topological recursion formalism, or via the ELSV formula (see alo \cite{Norbury} for a related approach). As a consequence of Corollary \ref{thm D} it is possible to further determine a recursion for colored Catalan numbers as a count of colored cell graphs with respect to edge contraction axioms, see \cite{thesis, DaDu}. 
%ie the number of ribbon graphs with 2 colored vertices on a Riemann surface of genus $g$. 

For Frobenius algebras, a recursion of twisted TQFT's via Catalan numbers 
is obtained by using Theorems \ref{thm A} and \ref{thm C} see \cite{DM1} together with the generalized Catalan recursion of \cite{DMSS}. This recursion of twisted TQFT's via Catalan  numbers can further be related to topological recursion, on the same spectral curve as the one determined by Catalan numbers, and has been shown in \cite{Serrano} to be related to enumerative geometry.
 
 For nearly Frobenius structures, Theorems \ref{thm A} and \ref{thm C} can be used to further determine  a recursion relation  in $A^{\tensor m}$ of Almost TQFTs weighted by the number of colored arrow cell graphs with respect to the corresponding contraction axioms of cell graphs.  It is an interesting question in the area of Topological Recursion to relate these recursions to enumerative geometry and Gromow-Witten theory establishing a rigorous mathematical foundation between enumerative geometry of ribbon graphs and invariants on moduli spaces of curves.
 
 It is also an interesting question to further extend the classification of Teleman of Cohomological field theories  \cite{Teleman_2011} together with its combinatorial interpretation via the theory of ribbon graphs as in \cite{DM3} for {\it nearly Frobenius structures}, by extending the classification Theorem \ref{thm C} from Almost TQFT to Cohomological Field Theories.

\subsection*{Acknowledgements} The authors would like to express their graditude to the Max Planck Institute for Mathematics, Bonn for their generosity, hospitality and stimulating environment during their long stay in 2023, when this project was initiated. They would also like to express their gratitude to the I.H.E.S and R.I.M.S for the visit in 2024. During their stay in MPIM, the second author gave a mini-course in University of Oxford, where some preliminaries ideas of this project were presented. 

The research of the authors was partially supported by the NSF-FRG DMS 2152130 grant and the UNC JFDA Award 2022.
\section{Preliminaries}\label{Section2}
%\subsection{Background}\label{background}
%\subsection{Frobenius Algebras}\label{secFA}
\noindent Throughout this chapter, we will consider finite-dimensional, unital, commutative Frobenius algebras over a field $K$.
\begin{Def}
    (Frobenius algebra) A \textbf{Frobenius algebra} is a pair $(A,\eta)$ where $A$ is a finite-dimensional, unital algebra over a field $K$ with a standard associative multiplication map $m:A\tensor A\to A$. The map $\eta:A\tensor A\to K$ is a non-degenerate symmetric bilinear form satisfying 
    \be
        \eta(u,vw)=\eta(uv,w),\quad \text{for\,all\,} u,v,w\in A.
    \ee 
    In this case, $\eta$ is known as the \textbf{Frobenius form}.
\end{Def}
\noindent Many computations with Frobenius algebras become more convenient when phrased in terms of a basis. If $\{e_1,e_2,\ldots, e_r\}$ is a $K$-basis for $A$, then we can write $\eta$ as a symmetric invertible matrix in terms of this basis.
\be\label{formasbasis}
\eta_{ij}:=\eta(e_i,e_j),\quad \eta=[\eta_{ij}],\quad \eta\inv=[\eta^{ij}]
\ee
%Say something about symmetric bilin form?
Given the Frobenius form $\eta$, we can define a counit $\eps: A\to K$. If $1\in A$ is the multiplicative identity of $A$, then for all $v\in A$, we can define the counit by $\eps(v)=\eta(1,v)$.
\begin{Rmk}
    Just as we can define the counit given a Frobenius form. We can conversely define the bilinear form based on the counit. Given a counit $\eps:A\to K$, we can define the form $\eta$ to be $$\eta(u,v)=\eps(uv),\,\text{for\,all\,} u,v\in A.$$ Because of this, some texts (particularly the one from Kock \cite{Kock}) refer to the counit as the Frobenius form and the bilinear form $\eta$ as the \textit{Frobenius pairing}.
\end{Rmk}
\noindent We further consider the map 
\be\label{eqn:lambdaiso}
\lambda:A\longrightarrow A^*,\quad \lambda(u)=\eta(u,-)
\ee
which, by the non-degeneracy of $\eta$, is an isomorphism from $A$ to its dual in the case that $A$ is finite dimensional. If we consider the dual of the multiplication map $m^{*}:A^{*}\longrightarrow A^{*}\tensor A^{*}$, then this isomorphism induces a unique comultiplication map $\delta:A\longrightarrow A\tensor A$ such that the following diagram commutes.
\be
	\begin{tikzcd}
		A \arrow[dd, "\lambda"'] \arrow[rr, "\delta", dashed] &  & A\otimes A \arrow[dd, "\lambda\otimes\lambda"] \\
		& \circlearrowleft &                                                \\
		A^* \arrow[rr, "m^*"']                                &  & A^*\otimes A^*                                
	\end{tikzcd}
\ee
We can explicitly compute this comultiplication given our basis of $A$: \be\label{eqn:deltabasis}
\delta(v):=\sum_{i,j,a,b=1}^{r}\eta(v,e_i e_j)\eta^{ia}\eta^{jb}e_a\otimes e_b\ee

%Character Theory Example
\begin{Ex}\label{ex:character}
     For a finite group $G$, we can consider the ring $R(G)$ of class functions $G\longrightarrow\cpx$, that is, functions which are constant on each conjugacy class of $G$. This ring is equipped with a bilinear form $\eta$ given by \be\eta(\varphi,\psi)=\dfrac{1}{|G|}\sum_{g\in G}\varphi(g)\psi(g\inv).\ee
    Then, the irreducible characters of $G$ for an orthonormal basis with respect to $\eta$ and therefore $\eta$ is a nondegenerate form that endows $R(G)$ with a commutative Frobenius algebra structure.\\
    \\
    Take for example $G=S_3$, the symmetric group on three elements. This group has three conjugacy classes, indexed by the cycle type of the permutation in the group: one conjugacy class containing only the identity permutation, one conjugacy class containing two $3$-cycles, and one conjugacy class containing three transpositions. The character table of this group is shown below.

    \begin{table}[h]
        \centering
        \begin{tabular}{|c|c|c|c|}
           \hline Conjugacy Class & Size $1$ & Size $2$ & Size $3$\\
           Representative & $\id$ & $(1\,\,2\,\,3)$ & $(1\,\,2)$\\\hline
            $\chi_1$ & $1$ & $1$ & $1$\\\hline
            $\chi_2$ & $1$ & $1$ & $-1$\\\hline
            $\chi_3$ & $2$ & $-1$ & $0$\\\hline
        \end{tabular}
        \vskip 10pt
        \caption{Character table of the symmetric group on $3$ elements}
        \label{tab:s3character}
    \end{table}

\noindent Here, $\chi_1$ is the character of the trivial representation, $\chi_2$ is the character of the dimension-$1$ alternating representation, and $\chi_3$ is the character of the dimension-$2$ irreducible representation. Considering the multiplication of two class functions that $(\chi\cdot\xi)(g)=\chi(g)\xi(g)$, we can see the multiplication on the basis is given by \begin{align*}
    \chi_1\chi_i=\chi_i,\quad\chi_2^2=\chi_1,\quad\chi_3^2=\chi_1+\chi_2+\chi_3,\quad\chi_2\chi_3=\chi_3.
\end{align*}
From this, we can compute the coproduct on basis elements. Since these characters form an orthonormal basis, then $\eta_{ij}=\eta^{ij}=1$ if $i=j$ and $0$ otherwise. Then, we see that \be
\delta(\xi)=\sum_{i,j=1}^3\eta^{ij}\xi\chi_i\tensor \chi_j=\sum_{i=1}^3 \xi \chi_i\tensor\chi_i,
\ee for any character $\xi$. For the irreducible characters, this gives that 
\begin{align*}
    \delta(\chi_1)&=\chi_1\tensor\chi_1+\chi_2\tensor\chi_2+\chi_3\tensor\chi_3,\\
\delta(\chi_2)&=\chi_2\tensor\chi_1+\chi_1\tensor\chi_2+\chi_3\tensor\chi_3,\\
\delta(\chi_3)&=\chi_3\tensor\chi_1+\chi_3\tensor\chi_2+(\chi_1+\chi_2+\chi_3)\tensor\chi_3.
\end{align*}
\end{Ex}
%Group Algebra Example
\begin{Ex}\label{ex:groupalg}
    For any finite group $G$, the group algebra $K[G]$ over $K$ can be given a Frobenius algebra structure. This algebra has basis $\displaystyle\bigsqcup_{g\in G}\{e_g\}$ with multiplication defined on the basis $e_ge_h=e_{gh}$. We then define the Frobenius form on this basis
    \be
        \eta(e_g,e_h)=\begin{cases}
            1 & \text{if\,} g=h\inv\\
            0 & \text{otherwise}
        \end{cases}
    \ee
    and extend linearly. Consider the order of this basis such that we list the basis element associated to the group identity, followed by all basis elements associated to order 2 elements of the group, followed by pairs of basis elements corresponding to inverse group elements. Then, we can write the matrix associated to $\eta$ as in \ref{formasbasis} as a block diagonal matrix featuring an identity block followed by several 2-by-2 blocks of with 0 on the diagonal and 1 off the diagonal. $$\eta=\left(\begin{array}{c|cc|c|cc}
       I_n & 0 & 0 & \cdots & 0 & 0\\\hline
         0 & 0 & 1 & \cdots & 0 & 0\\
         0 & 1 & 0 & \cdots & 0 & 0\\\hline
         \vdots & \vdots & \vdots & \ddots & \vdots & \vdots\\\hline
         0 & 0 & 0 & \cdots & 0 & 1\\
         0 & 0 & 0 & \cdots & 1 & 0\\
       
    \end{array}\right)$$
    Importantly, we notice that each nontrivial block along the diagonal of this matrix is its own inverse, so this matrix as a whole is its own inverse. In particular, $\eta_{gh}=\eta^{gh}$ for all $g,h\in G$.
\end{Ex}
\begin{Rmk}
    Note that if we instead define this form such that $\eta(e_g,e_h)=1$ for all basis elements $g,h\in G$, then this endows $K[G]$ with a Hopf algebra structure rather than a Frobenius algebra structure.
\end{Rmk}

%Cohomology example
\begin{Ex}\label{ex:cohom}
    Given an oriented compact differential manifold $M$, we can give the cohomology ring $H^*(M,\real)$ a Frobenius algebra structure. Here the multiplicative structure is just the cup product, and we define the counit to be the restriction to the top-dimensional cohomology. In particular, if $\dim_KM=n,$ then we define the counit to be the restriction $$\eps:H^*(M,\real)\longrightarrow H^n(M,\real).$$
\end{Ex}

\noindent Further, $A$ is both associative and coassociative, meaning that $(m(m(u\otimes v)\otimes w)=m(u\otimes m(v\otimes w))$ and $(\delta\otimes \id)(\delta(v))=(\id\otimes\delta)(\delta(v))$ respectively, and the maps $m$ and $\delta$ satisfy the ``Frobenius relation'' which is to say the following diagram commutes:
\be\label{cd:FrobRelDiag}
	\begin{tikzcd}
		& A\otimes A\otimes A \arrow[rd, "m\otimes 1"]  &            \\
		A\otimes A \arrow[ru, "1\otimes\delta"] \arrow[rd, "\delta\otimes1"'] \arrow[r, "m"] & A \arrow[r, "\delta"]                         & A\otimes A \\
		& A\otimes A\otimes A \arrow[ru, "1\otimes m"'] &           
	\end{tikzcd}
\ee
This is equivalent to saying that for all $u,v\in A$ that $\delta(uv)=(\id\tensor m)(\delta(u),v)=(m\tensor\id)(u,\delta(v)).$
In the case that we have defined the coproduct in terms of a basis for $A$,  the commutativity of this diagram follows from a general result for non-degenerate bilinear forms in which we can write any vector $v$ in terms of the basis: \be\label{eqn:canonbasis}
v=\sum_{a,b}\eta(v,e_a)\eta^{ab}e_b=\sum_{a,b}\eta(e_a,v)\eta^{ba}e_b.
\ee

As seen above, the composition $\delta\circ m:A\tensor A\longrightarrow A\tensor A$ plays an important role in the definition of a Frobenius algebra. This composition in the other order $m\circ \delta:A\longrightarrow A$ is important in the development of 2D TQFT.

\begin{Def}
(Euler element) The \textbf{Euler element} of a Frobenius algebra is a special element, which we denote $\eul$, defined by
\be
\eul = m\circ\delta(1).
\ee 
Where $1\in A$ is the multiplicative identity. In terms of a basis, we can see
\be\label{eqn:eulbasis}
\eul = \sum_{a,b}\eta^{ab}e_ae_b.
\ee
    This element allows us to extend our understanding of 2D TQFT to those linear maps associated to surfaces (or in our case, cell graphs) of positive genus.
\end{Def}
\begin{Ex}
    As we computed the coproduct of the ring of class functions in Example \ref{ex:character}, it follows that the Euler element is given by 
    \be
    \eul=m(\delta(\chi_1))=\chi_1^2+\chi_2^2+\chi_3^2=3\chi_1+\chi_2+\chi_3.
    \ee
    Writing $\eul$ in terms of the basis, we can see that the associated class function to the Euler element is not only constant on conjugacy classes, but non-zero on each conjugacy class. Therefore, the Euler element must be invertible and its inverse is the character that takes the inverse values on each conjugacy class.
        \begin{table}[h]
        \centering
        \begin{tabular}{|c|c|c|c|}
           \hline Conjugacy Class & Size $1$ & Size $2$ & Size $3$\\
           Representative & $\id$ & $(1\,\,2\,\,3)$ & $(1\,\,2)$\\\hline
            $\eul$ & $6$ & $3$ & $2$\\\hline
            $\eul\inv$ & $\frac{1}{6}$ & $\frac{1}{3}$ & $\frac{1}{2}$\\\hline
        \end{tabular}
        \label{tab:eulercharacter}
    \end{table}
    We can further compute using $\eta$ that \be
    \eul\inv=\frac{7}{18}\chi_1-\frac{1}{9}\chi_2-\frac{1}{18}\chi_3.
    \ee
\end{Ex}

\begin{Ex}
    Following the Example \ref{ex:groupalg} of the group algebra $K[G]$, we can see that $\eta^{gh}= 1$ if $g=h\inv$ and is zero otherwise. In this case, we see the Euler element can be computed 
    \begin{align*}
        \eul &= \sum_{g,h\in G}\eta^{gh}e_g e_h \\
        &= \sum_{g\in G}\eta^{gg\inv}e_ge_{g\inv}\\
        &= \sum_{g\in G}e_{1_G}\\
        &=|G| e_{1_G}.
    \end{align*}
\end{Ex}
\begin{Rmk}
    We also note that for any vector $v$, applying the coproduct followed by the product map is the same as simply multiplying by the Euler element. That is to say, 
    \be
    m\circ \delta(v)=\eul v.
    \ee
    This follows directly from \ref{eqn:deltabasis} and \ref{eqn:eulbasis}.
    \begin{align*}
        m\circ\delta(v) &= m\left(\sum_{i,j,a,b}\eta(v,e_i e_j)\eta^{ia}\eta^{jb}e_a\otimes e_b \right)\\
        &=\sum_{i,j,a,b}\eta(v,e_i e_j)\eta^{ia}\eta^{jb}e_a e_b\\
        &=\sum_{i,j,a,b}\eta(ve_i, e_j)\eta^{ia}\eta^{jb}e_a e_b\\
        &=\sum_{i,a}\eta^{ia} e_ie_av \\
        &=\eul v.
    \end{align*}
\end{Rmk}
\section{Frobenius algebras and TQFTs}\label{Frob}
\subsection{Topological Quantum Field Theory}\label{sec:TQFT}
In this section, we introduce the notion of topological quantum field theories (TQFT). While first introduced by Atiyah \cite{Atiyah} and refined by Segal \cite{Segal} in the late 1980s, Dijkgraaf \cite{Dijkgraaf} first discovered the connections between 2-dimensional TQFTs and Frobenius algebras. In the early 2000s, Moore and Segal \cite{MS} showed an equivalence of categories between the category of finite-dimensional Frobenius algebras and the category of 2-dimensional TQFTs. Throughout, we will refer to the fundamental literature of Kock \cite{Kock} and Teleman \cite{Teleman_2011} when discussing such connections.

\noindent We begin by considering the category \cob of 2-cobordisms. The objects of this category are disjoint copies of the circle $S^1$ while the morphisms are smooth oriented manifolds with boundary. In particular, if $M_0=\sqcup_n S^1$ and $M_1=\sqcup_m S^1$ are objects, then a morphism from $M_0$ to $M_1$ is a Riemann surface $\Sigma_{g,n,m}$ with boundary $\partial\Sigma_{g,n,m}=M_0^{op}\sqcup M_1$ where this means that the boundary components corresponding to $M_0$ have the opposite orientation of those corresponding to $M_1$.

%generators and relations for 2cob
\begin{Rmk}
    The morphisms of the category \cob are generated by the cobordisms corresponding to the unit, counit, identity map, product, and coproduct of the associated Frobenius algebra. The associativity and coassociativity relations of the Frobenius algebra can be regarded as the composition of pairs of pants with the same orientation, along with an identity morphism.
\end{Rmk}
\noindent The last relation here is what is known as the Frobenius relation, which deals with the composition of pairs of pants with opposite orientations. These are each diffeomorphic as Riemann surfaces of the same topological type with genus 0, two incoming boundary components, and two outgoing boundary components. This relation is equivalent to the commutativity of the diagram \ref{cd:FrobRelDiag}.

\begin{figure}[ht]
    \centering
    \includegraphics[width=0.8\linewidth]{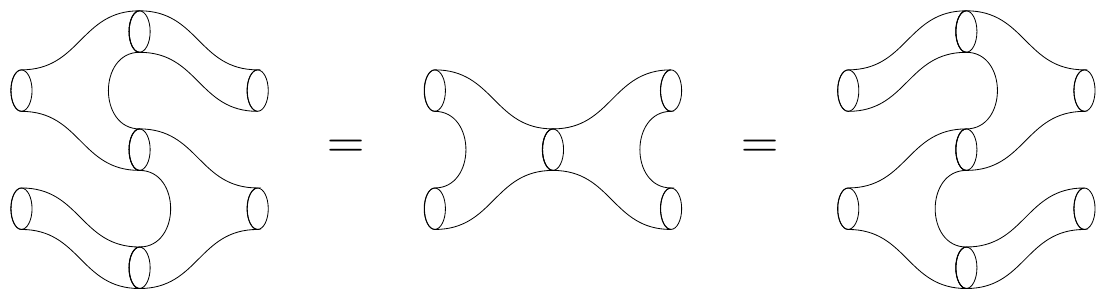}
    \caption{Frobenius relation of cobordisms.}
    \label{fig:FrobRel}
\end{figure}

\noindent Now, we fix a field $K$ and consider the monoidal category \textbf{KVect} of vector spaces over $K$ with a tensor product. A 2-dimensional topological quantum field theory (2D TQFT) is a symmetric monoidal functor $Z$ from the category \cob to the category \textbf{KVect} satisfying some particular axioms. For this given functor, we have that $Z(S^1)=A$ is a vector space over $K$ and $\omega_{g,n,m}:=Z(\Sigma_{g,n,m}):\Atn\longrightarrow\Atm$ is a multilinear map. The axioms of this functor and the functoriality respecting the relations on the generators of \cob makes $A$ necessarily a unital, commutative Frobenius Algebra.

\noindent We can further see how these relations on generators of \cob translate to properties of $A$. First, each of the generators corresponds to a particular piece of data defining A.
\begin{align*}
    \omega_{0,1,0}&:=Z(\Sigma_{0,1,0})=\eps:A\longrightarrow K,\\
    \omega_{0,0,1}&:=Z(\Sigma_{0,0,1})=\mathbf{1}:K\longrightarrow A,\\
    \omega_{0,2,0}&:=Z(\Sigma_{0,2,0})=\eta:A\tensor A\longrightarrow K,\\
    \omega_{0,1,2}&:=Z(\Sigma_{0,1,2})=\delta:A\longrightarrow A\tensor A,\\
    \omega_{0,2,1}&:=Z(\Sigma_{0,2,1})=m:A\tensor A \longrightarrow A.
\end{align*}

%define atiyah sewing/partial sewing
\noindent The \textbf{sewing axiom} of Atiyah and Segal \cite{Atiyah} states that the functor $Z$ respects the composition of morphisms in that the multilinear map corresponding to the composition of cobordisms is precisely the composition of the multilinear maps corresponding to the cobordisms in the composition. 

\noindent In particular, if $\Sigma_{g,n,k}$ is a genus $g$ cobordism with $n$ incoming boundary components and $k$ outcoming boundary components and $\Sigma_{h,k,m}$ is a genus $h$ cobordism with $k$ incoming boundary components and $m$ outgoing boundary components, then sewing them together at these $k$ pairs of boundary components creates a new cobordism $\Sigma_{g+h+k-1,n,m}$ with genus $g+h+k-1$, $n$ incoming boundary components and $m$ outgoing boundary components. The sewing axiom requires that the multilinear map associated to this new surface is exactly the composition of the maps corresponding to $\Sigma_{g,n,k}$ and $\Sigma_{h,k,m}$. That is,
\be\label{eqn:sewing}
\omega_{h,k,m}\circ\omega_{g,n,k}=\omega_{g+h+k-1,n,m}:\Atn\longrightarrow\Atm.
\ee
\begin{figure}[ht]
    \centering
    \includegraphics[width=0.3\linewidth]{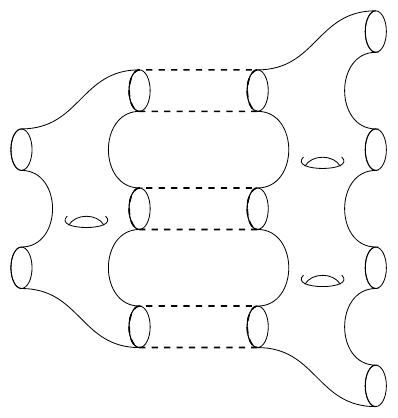}
    \caption{Sewing two cobordisms with agreeing numbers of boundary components.}
    \label{fig:sewing}
\end{figure}

\noindent This sewing axiom can be generalized into a \textbf{partial sewing} axiom in which we sew together only some subset of these boundary components. Suppose $\Sigma_{g,n,\ell}$ and $\Sigma_{h,k,m}$ are two cobordisms. Then, for any $j\geq 1$ with $j\leq \ell$ and $j\leq k$, we can sew together $j$ of the outgoing boundary components of $\Sigma_{g,n,\ell}$ with $j$ of the incoming boundary components of $\Sigma_{h,k,m}$, and sew identity cylinders to the other appropriate boundary components. This creates a new cobordism $\Sigma_{g+h+j-1,n+k-j, \ell+m-j}$, and the partial sewing axiom further requires that the TQFT respects this composition. That is, 
\be\label{eqn:partialsewing}
\omega_{h,k,m}\circ\omega_{g,n,\ell}=\omega_{g+h+j-1,n+k-j,m+\ell-j}:A^{\tensor n+k-j}\longrightarrow A^{m+\ell-j}.
\ee
%defn of tqft as a system of maps
This partial sewing axiom allows us to see relationships between the canonical maps of the Frobenius algebra structure. For example,
\be
\omega_{0,1,0}\circ \omega_{0,2,1}=\omega_{0,0,2}\Longrightarrow \eps\circ m=\eta.
\ee

\noindent Atiyah's work introduces one more axiom regarding duality and the reversal of orientations. In particular, if $Z(S^1)=A$, then $Z((S^1)^{op})=A^*$ and if $\Sigma$ is a cobordism of type $(g,n,m)$ and $\Sigma^*$ is the surface $\Sigma$ with the opposite orientation, then $Z(\Sigma^*)$ gives the dual map on dual vector spaces $(A^*)^{\tensor n}\longrightarrow (A^*)^{\tensor m}$. This gives $A$ a bialgebra structure. To guarantee duality between the algebra and coalgebra structures of $A$, we require that $\eta$ is nondegenerate. This means that the map $\lambda$ in \ref{eqn:lambdaiso} is an isomorphism, and makes $A$ a commutative Frobenius algebra.

\noindent We can see explicitly that the functoriality of $Z$ applied to the disk-sewing relation \be
(\omega_{0,1,0}\tensor \id)\circ \omega_{0,1,2}=\omega_{0,1,1}
\ee
gives the canonical basis expansion
\ref{eqn:canonbasis}. In particular, we can see \begin{align*}
    v&=(\id\tensor\eps)\circ\delta(v)\\
    &=\sum_{i,a}\eta(v,e_i)e_a.
\end{align*}

Because the relations on the generators of \cob are analogous to properties inherent to Frobenius algebras, we can conversely define a 2D TQFT starting with the generating maps $1,\eps,\id, m, $ and $\delta$, build all other maps from the partial sewing axiom, subject to a few restrictions.

\begin{Def}
    (2-Dimensional Topological Quantum Field Theory, \cite{Atiyah} ) Let $(A,\eps,\lambda)$ be the following set of data. $A$ is a finite-dimensional vector space over a field $K$. $\eps$ is a non-trivial linear map $A\longrightarrow K$, and $\lambda: A\stackrel{\sim}{\longrightarrow} A^*$ is an isomorphism from $A$ to its dual. A \textbf{2-dimensional topological quantum field theory} is a system of multilinear maps $\left(A,\left\{\omega_{g,n,m}\right\}\right)$ where $$\omega_{g,n,m}:\Atn\longrightarrow\Atm,\quad g,n,m\geq0$$satisfy certain axioms related to symmetry with respect to the symmetric group action on its domain, non-triviality of the counit, duality with respect to reversing the orientation of cobordisms, and the partial sewing axiom.
   % \begin{itemize}
   %     \item \textbf{TQFT 1. Symmetry:} The map \be
   %     \omega_{g,n,m}: \Atn\longrightarrow \Atm 
  %      \ee
  %      is invariant under the symmetric group action on its domain.
  %      \item \textbf{TQFT 2. Non-triviality: }\be
  %      \omega_{0,1,0}:=\eps:A\longrightarrow K.
  %      \ee

    %    \item \textbf{TQFT 3. Duality: }The following diagram commutes:
     %   \be
%\begin{tikzcd}
%A^{\otimes n} \arrow[dd, "\lambda^{\otimes n}"'] \arrow[rr, "{\omega_{g,n,m}}"] &  & A^{\otimes m} \arrow[dd, "\lambda^{\otimes m}"] \\
%                                                                                &  &                                                 \\
%(A^*)^{\otimes n} \arrow[rr, "{(\omega_{g,n,m})^*}"]                            &  & (A^*)^{\otimes m}                              
%\end{tikzcd}
%        \ee

%    \item \textbf{TQFT 4. Partial Sewing: }\be
%\omega_{h,k,m}\circ\omega_{g,n,\ell}=\omega_{g+h+j-1,n+k-j,m+\ell-j}:A^{\tensor n+k-j}\longrightarrow A^{m+\ell-j},\quad j\leq k,j\leq\ell.
 %   \ee
 %   \end{itemize}
\end{Def}

\noindent For any 2D TQFT satisfying these axioms, we can give a classification of these maps $\omega_{g,n,m}$. Paper \cite{DM1} gives a proof of this classification for the maps of the form $\omega_{g,n,0}$ and suggest how to use duality to reconstruct the $\omega_{g,n,m}$ maps in more generality. In this section for Frobenius algebras, we will be more explicitly and show this classification of the maps $\omega_{g,n,m}$ in the general case. This proof uses axioms governing the nature of Cohomological Field Theories. 
\subsection{Cohomological Field Theories.} Let $\overline{\mathcal{M}_{g,n}}$ denote the Deligne-Mumford compactification of the moduli space of stable curves of genus $g$ and $n$ marked points.
\noindent  A \textbf{Cohomological Field Theory} (CohFT) introduced by Kontsevich and Manin in \cite{KM} is a Frobenius algebra $A$ along with a system of linear maps $\Omega_{g,n}:A^{\tensor n}\longrightarrow H^*(\overline{\mathcal{M}_{g,n}},K)$ defined in the stable range for $2g-2+n>0$, satisfying a list of axioms so that $\Omega_{g.n}$ is compatible with the with the boundary maps of the moduli space $\overline{\mathcal{M}_{g,n}}$ gluing singular stable curves and the forgetful morphism.  Proposition 3.4 of \cite{DM1} implies that the $\left\{\omega_{g,n,0}\right\}$ part of any 2D TQFT gives a CohFT that takes only values in $H^0(\overline{\mathcal{M}_{g,n}},K)=K$ and conversely by Proposition 3.3 of \cite{DM1}, the restriction of any CohFT to the degree $0$ part of the cohomology ring gives the $\left\{\omega_{g,n,0}\right\}$ part of a 2D TQFT . In the context of the 2D TQFT maps $\left\{\omega_{g,n,0}\right\}$, the CohFT axioms become the following:
\be\label{eqn:cohft1}
\omega_{g,n+1,0}(v_1,\ldots,v_n,1)=\omega_{g,n,0}(v_1,\ldots,v_n).
\ee
\be\label{eqn:cohft2}
\omega_{g,n,0}(v_1,\ldots,v_n)=\sum_{a,b}\omega_{g-1,n+2,0}(v_1,\ldots,v_n,e_a,e_b)\eta^{ab}.
\ee
\be\label{eqn:cohft3}
\omega_{g_1+g_2,|I|+|J|,0}(v_I,v_J)=\sum_{a,b}\eta^{ab} \omega_{g_1,|I|+1,0}(v_I,e_a)\cdot \omega_{g_2,|J|+1,0}(v_J,e_b).
\ee
Where here $I\sqcup J=\{1,2,\ldots,n\}$. As it turns out, each of these three equations is a direct result of the partial sewing axiom \ref{eqn:partialsewing}. \noindent We now recall the following result on the classification of 2D TQFT , see \cite{DM2} Theorem 3.8 or \cite{DM1}, Theorem 3.9 of \cite{DM1} for $m=0$.
%give the folklore theorem about delta^m-1, proof from DM. First do g,n,0 then reconstruct?

\begin{Thm}[\cite{DM1,DM2}]\label{prop:gn0case}
    For maps whose outputs are in $K$, we have
    \be\label{eqn:gn0case}
    \omega_{g,n,0}(v_1,\ldots,v_n)=\eps(v_1\cdots v_n\eul^g).
    \ee
\end{Thm}

\noindent We now can reconstruct the classification of 2D TQFT by appealing to duality. Note that a linear map $\omega_{g,n+m,0}:\Atn\tensor\Atm\longrightarrow K$ is equivalent to a map $\Atn\longrightarrow (A^*)^{\tensor m}$, so we can reconstruct the map $\omega_{g,n,m}$ from $\omega_{g,n+m,0}$ using the following diagram:
\be\label{cd:reconstruct}
\begin{tikzcd}
A^{\otimes n} \arrow[dd,"="', rotate = 90] \arrow[rr, "{\omega_{g,n+m,0}}"] &  & (A^*)^{\otimes m} \arrow[rr, "(\lambda\inv)^{\otimes m}"] &  & A^{\otimes m} \arrow[dd, "="] \\
                                                                &  &                                                           &  &                               \\
A^{\otimes n} \arrow[rrrr, "{\omega_{g,n,m}}"]                  &  &                                                           &  & A^{\otimes m}                
\end{tikzcd},
\ee
where $\lambda:A\stackrel{\sim}{\longrightarrow}A^*$ is the isomorphism given by $\lambda(v)=\eta(v,-).$ 
\subsection{A classification result}
Using the ideas developed in previous sections, we can prove the following classification result for $m>0$ noting that the $m=0$ case was discussed in \cite{DM2}.

    \begin{Thm}[Theorem \ref{thm B}]\label{prop:tqftclassification}
      If $A$ is a Frobenius algebra, then the value of a 2D TQFT is given by
        \be\label{eqn:tqftclassification}
        \omega_{g,n,m}(v_1,\ldots,v_n)=\begin{cases}
             v_1\cdots v_n\eul^g,& m=1\\                \eps(\omega_{g,n,1}(v_1,\ldots,v_n)),& m=0\\
            \delta^{m-1}(\omega_{g,n,1}(v_1,\ldots,v_n)),& m\geq 2
        \end{cases}.
        \ee
    \end{Thm}

\begin{proof}

   \noindent By Theorem \ref{prop:gn0case}, it suffices to  prove the theorem for $m>0$. We now observe the process of reconstructing maps with a positive number of outputs to give a classification of the entire system of maps $\left\{\omega_{g,n,m}\right\}$ in a 2D TQFT. We first begin by reconstructing $\omega_{g,n,1}$ from $\omega_{g,n+1,0}$. We see by the commutativity of \ref{cd:reconstruct} that $\omega_{g,n,1}=\lambda\inv\circ\omega_{g,n+1,0}.$ Specifically, we see that \begin{align*}
       \lambda\inv\circ\omega_{g,n+1,0}(v_1,\ldots,v_n,-)&=\lambda\inv(\eps(v_1\cdots v_n-)\\
       &=\lambda\inv(\eta(v_1\cdots v_n,-))\\
       &=v_1\cdots v_n.
   \end{align*}
   \noindent To see how to reconstruct the maps for $m>1$, we first must discuss the possible forms that successive comultiplication can take. Applying $\delta$ to the first tensor component of the formula \ref{eqn:deltabasis}, we see that \begin{align*}
    (\delta\tensor \id)(\delta(v))&=(\delta\tensor\id)\left(\sum_{a,b,i,j}\eta(v,e_ie_j)\eta^{ia}\eta^{jb}e_a\tensor e_b\right)   \\
    &= (\delta\tensor\id)\left(\sum_{a,b,i,j}\eta(ve_j,e_i)\eta^{ia}\eta^{jb}e_a\tensor e_b\right) \\
    &=(\delta\tensor\id)\left(\sum_{b,j}ve_j\tensor \eta^{jb}e_b\right)\\
    &=\sum_{b,j,c,d,k,\ell}\eta(ve_j,e_ce_d)\eta^{ck}\eta^{d\ell}e_k\tensor e_\ell\tensor \eta^{jb}e_b\\
    &=\sum_{b,j,c,d,k,\ell}\eta(ve_je_d,e_c)\eta^{ck}e_k\tensor \eta^{d\ell}e_\ell\tensor \eta^{jb}e_b\\
    &=\sum_{b,d,j,\ell}ve_je_d\tensor \eta^{d\ell}e_\ell\tensor \eta^{jb}e_b.
   \end{align*}
    If we continue this process of applying the coproduct to the first component and the identity to all other components, then up to some relabeling we see that \be\label{eqn:deltapowers}
    \delta^{m-1}(v)=\sum_{\substack{a_1,\ldots,a_m\\b_1,\ldots,b_m}}ve_{a_1}e_{a_2}\cdots e_{a_m}\tensor \eta^{a_1b_1}e_{b_1}\tensor \cdots\tensor \eta^{a_mb_m}e_{b_m}\in A^{\tensor m}.
    \ee
    In the case of \ref{eqn:deltapowers}, it is clear that by $\delta^{m-1}(v)$ we mean $(\delta\tensor\id\tensor\cdots\tensor\id)\circ\cdots\circ(\delta\tensor\id)\circ\delta(v),$ but we note that by coassociativity, this formula is far from unique. There are many equivalent ways in which we can successively apply the coproduct, and as such many equivalent forms that this sum of tensors may take. We now that this expansion of $\delta^{m-1}(v)$ in \ref{eqn:deltapowers} and apply $\lambda^{\tensor m}$. This should give us an element of $(A^*)^{m}.$ 

    \noindent Here, we see that

    \begin{align*}
        \lambda^{\tensor m}(\delta^{m-1}(v))(w_1,\ldots, w_m)&=\lambda^{\tensor m}\left(\sum_{\substack{a_1,\ldots,a_m\\b_1,\ldots,b_m}}ve_{a_1}e_{a_2}\cdots e_{a_m}\tensor \eta^{a_1b_1}e_{b_1}\tensor \cdots\tensor \eta^{a_mb_m}e_{b_m}\right)(w_1,\ldots,w_{m})\\
        &=\sum_{\substack{a_1,\ldots,a_m\\b_1,\ldots,b_m}}\eta(ve_{a_1}\cdots e_{a_m},w_1) \eta^{a_1b_1}\eta(e_{b_1},w_2) \cdots\eta^{a_mb_m}\eta(e_{b_m},w_{m})\\
        &=\sum_{\substack{a_1,\ldots,a_m\\b_1,\ldots,b_m}}\eta(ve_{a_1}\cdots e_{a_m}\eta^{a_mb_m}\eta(e_{b_m},w_{m}),w_1) \eta^{a_1b_1}\eta(e_{b_1},w_2) \cdots\eta^{a_{m-1}b_{m-1}}e_{b_{m-1}}\\
        &=\sum_{\substack{a_1,\ldots,a_{m-1}\\b_1,\ldots,b_{m-1}}}\eta(ve_{a_1}\cdots e_{a_{m-1}}w_m,w_1) \eta^{a_1b_1}\eta(e_{b_1},w_2) \cdots\eta^{a_{m-1}b_{m-1}}e_{b_{m-1}}\\
    \end{align*}
    We continue this process, applying the canonical basis expansion \ref{eqn:canonbasis} ,to reduce the number of basis vectors each time by 2.
    \begin{align*}
       \lambda^{\tensor m}(\delta^{m-1}(v))(w_1,\ldots, w_m) &=\sum_{a_1,b_1}\eta(ve_{a_1},w_1w_2\cdots w_m)\eta^{a_1b_1}e_{b_1}\\
        &=\eta(v,w_1w_2\cdots w_m)\\
        &=\eps(vw_1w_2\cdots w_m)\\
        &=\omega_{0,m+1,0}(v,w_1,\ldots,w_m).
    \end{align*}
    \noindent More generally, we see that \be\lambda^{\tensor m}(\delta^{m-1}(v_1v_2\cdots v_n\eul^g)(w_1,w_2,\ldots, w_m)=\eps(v_1v_2\cdots v_n\eul^g w_1w_2\cdots w_m),\ee
    \noindent and it follows that \be
    (\lambda\inv)^{\tensor m}\omega_{g,n+m,0}(v_1,\ldots,v_n,-,\ldots,-)=\omega_{g,n,m}(v_1,\ldots,v_n)=\delta^{m-1}(v_1\cdots v_n\eul^g).
    \ee

\end{proof}
%relate to normal form of surface?

\section{Nearly Frobenius algebras}\label{Section3}

\begin{Def} (Nearly Frobenius Algebra) A \textbf{Nearly Frobenius Algebra} is a set of data $(A,m,\delta)$ where $A$ is an algebra over a field $K$ with multiplication map $m:A\tensor A\to A$, and the comultiplication $\delta:A\to A\tensor A$ is a bimodule map satisfying the Frobenius relation that \ref{cd:FrobRelDiag} commutes. We note that unlike a Frobenius algebra, $A$ is neither required to be unital nor counital.
\end{Def}

\noindent Recall that for finite-dimensional Frobenius algebras, the Euler element $\eul$ is defined as $\eul=(m\circ\delta)(1)$, and similarly $(m\circ\delta)(v)=\eul v$. For non-unital Nearly Frobenius algebras, however, $\eul$ cannot exist as an element of the algebra since the algebra has no unit. Instead, we must define an \textbf{Euler map}, $\bige$ defined as \be \bige:=m\circ\delta:A\longrightarrow A.\ee
We can still write this map in terms of a basis for $A$, 
\be\label{eqn:EULbasis}
\bige(v)=\sum_{a,b}\eta^{ab}ve_ae_b.
\ee
 In the finite dimensional case, then the Euler map acts as multiplication by the Euler element $e$. We then use this Euler map to classify values of the Almost TQFT for higher genus.

\subsection{On Ribbon TQFTs}
In this work we will not present the details of ribbon graphs on Riemann surfaces of genus $g$ and $n$ marked points. We recall that for Frobenius algebra, or finitely generated algebras over $K$, one of the main results of \cite{DM1} is that the 2D TQFTs can be equivalently defined on the the set (category) or ribbon graphs with edge contraction axioms. Indeed, to any cell graph of genus $g$ with $n$ vertices we can associate a map, that we call ribbon TQFT, from $A^{\tensor n}$ to $K$, so that the contracting edge operations on cell graphs are compatible to the multiplication in the Frobenius algebra $m$ and the loop contraction is compatible with the comultiplication $\delta$, see \ref{eqn:deltabasis}, of the Frobenius algebra \cite{DM2}. The list of axioms of such map imply that The Ribbon TQFT is independent of the ribbon graph, it depends only on the topological information of the riemann surface, namely the genus $g$ and the number of marked points $n$. In particular, for a Frobenius algebra the classification of the Ribbon TQFT give the same result as Theorem \ref{eqn:gn0case}. As a corollary one obtains that the Atiyah - Segal definition of 2D TQFTs for Frobenius algebras is equivalent to ribbon TQFT of \cite{DM2}. 

Moreover, the number of cell graphs drawn on a riemann surface of genus $g$ with $n$ marked points is an infinite number. However imposing that every vertex has a fixed number of half-edges adjacent to it, makes this count finite. To avoid symmetry, in \cite{DMSS} the authors are counting the number of cell graphs of type $(g,n)$ with fixed degree at each verted $\mu_i$ and one outgoing arrow attached to every edge and denote this number by $C_{g,n}(\mu_1,\ldots, \mu_n)$, namely a generalized Catalan number. These numbers, similar to Hurwitz numbers, satisfy a recursion formula that can be proved via the edge contraction axioms. The generating function of Catalan numbers, encode the Witten-Kontsevich intersection numbers of $\overline{\mathcal{M}_{g,n}}$.

For almost TQFTS defined on nearly Frobenius algebras the authors discovered an analogous method to define a Ribbon TQFTs to any colored cell graph of topological type $(g,n)$ so that the colored edge contraction operations are compatible with the multiplication and the comultiplication of the Nearly Frobenius algebra. Namely, the infinite dimensionality forces us to consider cell graphs with vertices labeled by two colors say $red$ and $blue$. This is because $\lambda$ is no longer an isomorphism between $A$ and $A^*$, and the color red labels the input while the color blue labels the output.
We recall the main theorem the authors obtained in \cite{thesis} and \cite{DaDu}

\begin{Thm}\label{thm A}

Let $A$ be a Nearly Frobenius algebra with counit $\eps$, coproduct $\delta$, and Euler map $\bige$. Further let $\gamma$ be a colored cell graph of type $(g,n,m)$ with $n,m\geq 1,$ and $\Omega(\gamma)\in \Hom(A^{\tensor n}, A^{\tensor m})$ an assignment of a multilinear map satisfying  the Edge Contraction Axioms \cite{thesis}. Then the {\it ribbon TQFT} associated to $\gamma$ is given by  
    \be
    \Omega(\gamma)(v_1,\ldots,v_n)=\delta^{m-1}(\bige^g(v_1\cdots v_n)),
    \ee
    where $\delta^0(v)=\id(v)$.
    
If $A$ is finite dimensional and $m=0$ then
$\Omega(\gamma)(v_1,\ldots,v_n)=
\eps(e^g\cdot v_1\cdots v_n))$.

\end{Thm}

\noindent We note here that the $m=0$ case is can be recovered from \cite{DM2}. 

\subsection{Frobenius structures}

It is well understood that any Frobenius algebra that is both unital and counital and whose bialgebra structure is both associative and coassociative is necessarily finite-dimensional. We would like to extend our theories and classifications to algebras which retain most of the same Frobenius structure, but may potentially be infinite dimensional. To do so, one must sacrifice either unitality or counitality. We note that Frobenius algebras can be defined based on the counit $\eps$ (see \cite{DaDu}, \cite{thesis}). However, Abrams \cite{Abrams} demonstrated that Frobenius algebras could be equivalently formulated based on the comultiplication $\delta$, in the sense that there is a one-to-one correspondense between Frobenius algebra structures defined by $\eps$ and those structures which are defined by $\delta$. Gonzalez et al. \cite{GLSU} then define this generalization of Frobenius algebras based on this comultiplicative structure.

To generalize 2D TQFT, Gonzalez et al. \cite{GLSU} introduce the idea of an \textit{Almost TQFT}, defined as a functor \be
 Z:\mathbf{2Cob^+}\longrightarrow\mathbf{\textbf{KVect}^\infty},
\ee
\noindent where $\mathbf{2Cob^+}$ is the full subcategory of \cob whose objects are a disjoint union of a positive number of copies of $S^1$ and $\mathbf{\textbf{KVect}^\infty}$ is the category of possibly infinite-dimensional vector spaces over $K$. While it is largely understood that there is an isomorphism of categories between the category of Frobenius algebras and 2D TQFT, they similarly show that there is an isomorphism of categories between the category of Nearly Frobenius algebras and the category of Almost TQFT.

%\begin{Ex}
%    {\color{blue} Kexuan's example about field extensions, but with $\mathbb{Q}\hookrightarrow\overline{\mathbb{Q}}$}
%\end{Ex}
%examples examples examples
\begin{Ex}
    We first note that any Frobenius algebra $A$ with counit $\eps$ is necessarily a Nearly Frobenius algebra. As in the previous section, the coproduct $\delta$ induced as the dual of the product map $m$ is precisely the same $\delta$ which is the coproduct of the Nearly Frobenius structure.
\end{Ex}
\noindent The converse of this example, however, is not necessarily true. That is, there are many Nearly Frobenius algebra structures which are not Frobenius algebras, even if these algebras are finite-dimensional vector spaces. Of the many examples of this, one follows from the idea that while Frobenius algebras have a counit which is essentially unique, a Nearly Frobenius algebra may have a whole space of comultiplicative structures that make is a Nearly Frobenius algebra.
\begin{Ex}

If $A$ and $B$ are both Nearly Frobenius algebras with coproducts $\delta_A$ and $\delta_b$, then $A\tensor B$ is also a Nearly Frobenius algebra with coproduct $\delta_A\tensor\tau\tensor\delta_B$ where $\tau$ is the transposition map swapping two entries. This is given by Theorem 3.4 of \cite{GLSU}.
\end{Ex}

\noindent As we noted before, Frobenius algebras must be finite-dimensional, but Nearly Frobenius algebras can be infinite dimensional.

\begin{Ex}
    Consider $A=K[[x,x\inv]]$, the algebra of formal Laurent series over $K$. Then, coproducts that look like linear combinations of 
    \be
    \delta_k(x^\ell)=\sum_{i+j=k+\ell}x^i\tensor x^j
    \ee
    \noindent all define Nearly Frobenius algebras which do not arise as counits of Frobenius algebras. For each of these coproducts, we can observe that their higher powers must look like
    \begin{align*}
        \delta_k^2(x^\ell)&=\sum_{i+j=k+\ell}\delta_k(x^i)\tensor x^j\\
        &=\sum_{i+j=k+\ell}\left(\sum_{a+b=k+i}x^a\tensor x^b\right)\tensor x^j\\
        &=\sum_{a+b+j-k=k+\ell}x^a\tensor x^b\tensor x^j\\
        &=\sum_{a+b+j=2k+\ell}x^a\tensor x^b\tensor x^j,
    \end{align*}or more generally, 
    \be
    \delta_k^{m}(x^\ell)=\sum_{a_1+a_2+\cdots a_{m-1}=mk+\ell}x^{a_1}\tensor \cdots \tensor x^{a_{m-1}}.
    \ee
\end{Ex}

%Remark about how this is the cohomology of something???

%possibly say that GLSU classified all semisimple FLA's?

\section{Atiyah Sewing for Counital Nearly Frobenius Algebras}\label{Section4}
 In this section, we would like to prove a similar classification result for Almost TQFT and potentially infinite-dimensional Nearly Frobenius algebras. These algebras are not unital (and often not counital), so the element $\delta(1)$ no longer exists in this structure. We can, however, provide a similar proof by sewing and removing the cobordism $\Sigma_{0,2,0}$ in the case that the corresponding Nearly Frobenius algebra still has a counit. Throughout the remainder of this section, we will suppose $A$ is a possibly infinite-dimensional Nearly Frobenius algebra with counit $\eps$.

\noindent We first introduce two useful results of the Atiyah-Segal partial sewing axiom involving sewing the tube-like shape $\Sigma_{0,2,0}$ associated to the map $\eta=\omega_{0,2,0}=\omega_{0,1,0}\circ\omega_{0,2,1}=\eps\circ m$. 
\begin{Lem}\label{lem:sewonright}
    The following recursive relationships hold for any counital Almost TQFT
    \be\label{eqn:SOR1}
    \omega_{g,n,m}(v_1,\ldots,v_n)=(\id^{\tensor n}\tensor \eta)\circ\omega_{g-1,n,m+2}(v_1,\ldots,v_n)
    \ee
    \be\label{eqn:SOR2}
    \omega_{g_1+g_2,n_1+n_2,m_1+m_2}(v_I,v_J)=(\id^{\tensor n_1-1}\tensor\eta\tensor\id^{\tensor n_2-1})\circ(\omega_{g_1,n_1,m_1+1}(v_I)\tensor\omega_{g_2,n_2,m_2+1}(v_J))
    \ee
    \noindent where $I\sqcup J=\{1,2,\ldots,n\}, |I|=n_1,$ and $|J|=n_2.$
\end{Lem}
\begin{figure}[ht]
    \centering
    \includegraphics[width=0.3\linewidth]{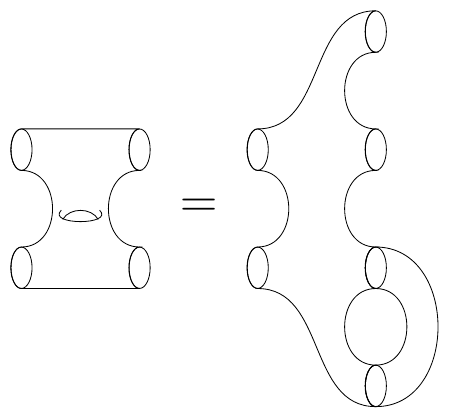}\hspace{3cm}\includegraphics[width=0.25\linewidth]{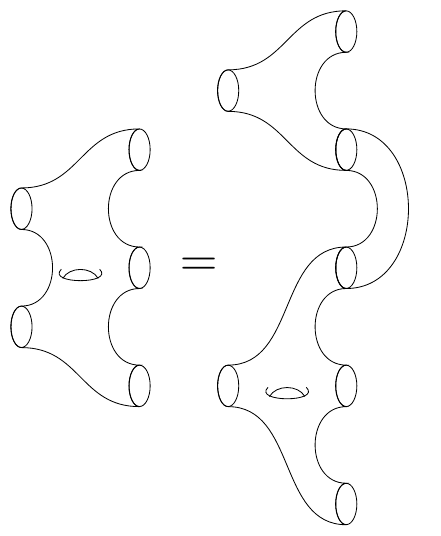}
    \caption{Partial sewing of $\Sigma_{0,2,0}$ onto the output side of a cobordism.}
    \label{fig:sewonright1}
\end{figure}
%\begin{figure}[ht]
 %   \centering
 %   \includegraphics[width=0.2\linewidth]{Figures/sewonright2.pdf}
 %   \caption{Partial sewing of $\Sigma_{0,2,0}$ connecting two disjoint cobordisms.}
 %   \label{fig:sewonright2}
%\end{figure}
\noindent We will use these two recursions to reconstruct the classification theorem for Almost TQFT. Note that since $A$ is not unital, then throughout this classification, we require that we only consider cobordisms where $n\geq 1$. As before, we want to consider cobordisms that fall in the stable range $2g-2+(n+m)>0$, so we take as base cases the cobordisms with genus 0 and 3 boundary components
\be
\omega_{0,1,2}(v)=\delta(v),\quad \omega_{0,2,1}(v,w)=vw,\quad \omega_{0,3,0}(u,v,w)=\eps(uvw),
\ee
\noindent and extend to the unstable range with $\omega_{0,1,0}(v)=\eps(v)$, $\omega_{0,1,1}(v)=v,$ and $\omega_{0,2,0}(v,w)=\eta(v,w)$.
\begin{Lem}\label{lem:SOR0n1}
For maps with genus 0 and one output, we have
\be\label{eqn:SOR0n1}
\omega_{0,n,1}(v_1,\ldots,v_n)=v_1\cdots v_n.
\ee
\end{Lem}
\begin{proof}
We proceed by induction on $n$, starting with the base case that $\omega_{0,2,1}(v_1,v_2)=v_1v_2$. From here, assuming that $\omega_{0,n-1,1}(v_1,\ldots,v_{n-1})=v_1\cdots v_{n-1}$, applying \ref{eqn:SOR2} and the canonical basis expansion \ref{eqn:canonbasis}, we see that 
\begin{align*}
    \omega_{0,n,1}(v_1,\ldots,v_n)&=(\omega_{0,2,0}\tensor \id)\circ(\omega_{0,n-1,1}(v_1,\ldots,v_{n-1})\tensor\omega_{0,1,2}(v_n))\\
    &=(\eta\tensor\id)(v_1\cdots v_{n-1}\tensor \delta(v_n))\\
    &=(\eta\tensor\id)\left(\sum_{a,b}v_1\cdots v_{n-1}\tensor v_n e_a\tensor \eta^{ab}e_b\right)\\
    &=\sum_{a,b}\eta(v_1\cdots v_{n-1},v_n e_a)\tensor \eta^{ab}e_b\\
    &=\sum_{a,b}\eta(v_1\cdots v_{n-1}v_n, e_a)\eta^{ab}e_b\\
    &=v_1\cdots v_{n-1}v_n.
\end{align*}
\end{proof}

\noindent
%FIGURE GOES HERE
We similarly use \ref{eqn:SOR2} to establish the classification result for maps with no outputs; however, this time we use $\Sigma_{0,2,0}$ to connect two cobordisms which each have only one output.
\begin{Lem}\label{lem:SOR0n0}
    For maps with genus 0 and outputs in $K$, we have
    \be\label{eqn:SOR0n0}
    \omega_{0,n,0}(v_1,\ldots,v_n)=\eps(v_1\cdots v_n).
    \ee
\end{Lem}
\begin{proof}
    This now follows as a result of \ref{eqn:SOR2} and Lemma \ref{lem:SOR0n1}.
    \begin{align*}
        \omega_{0,n,0}(v_1,\ldots,v_n)&=\eta\circ(\omega_{0,n-2,1}(v_1,\ldots,v_{n-2})\tensor\omega_{0,2,1}(v_{n-1},v_n))\\
        &=\eta(v_1\cdots v_{n-2},v_{n-1}v_n)\\
        &=\eps(v_1\cdots v_{n-2}v_{n-1}v_n).
    \end{align*}
\end{proof}
\noindent We now observe how to classify maps with more than one output, drawing inspiration from our reconstruction in \ref{cd:reconstruct}. One important thing to note is that if $A$ is no longer finite dimensional, then the map $\lambda:A\longrightarrow A^*$ given by $\lambda(v)=\eta(v,-)$ is no longer an isomorphism. It is, however, still injective onto its image, so instead of considering the inverse $\lambda\inv:A^*\longrightarrow A$, we would like to consider the map $r:\lambda(A)\longrightarrow A$ that represents the inverse of $\lambda$ wherever it is defined on the dual space $A^*$.

\noindent From the cobordism perspective, in the finite dimensional case, we could reconstruct $\Sigma_{g,n,m}$ by taking $\Sigma_{g,n+m,0}$ and sewing on $m$ copies of $\Sigma_{0,0,2}$ onto one input slot for each copy of $\sigma_{0,0,2}$, effectively turning these $m$ inputs into outputs. Now that we require our cobordisms to have at least one input slot, we can no longer sew copies of $\Sigma_{0,0,2}$. Instead, we realize $\Sigma_{0,n+m-1,1}$ as the result of sewing $m-1$ copies of $\Sigma_{0,2,0}$ onto $m-1$ of the output slots of $\Sigma_{0,n,m}$ and therefore $\Sigma_{0,n,m}$ can be realized by the removal of these copies of $\Sigma_{0,2,0}$ from $\Sigma_{0,n+m-1,1}$.
\begin{Lem}\label{lem:SOR01m}
    For maps with genus 0 and $m\geq 1$, the value of a counital Almost TQFT is given by
    \be\label{eqn:SOR01m}
    \omega_{0,1,m}(v)=\delta^{m-1}(v).
    \ee
\end{Lem}
\begin{proof}
    We first begin by showing that $\omega_{0,m,1}=(\lambda^{\tensor m-1}\tensor \id)\circ\delta^{m-1}$. To that end, using \ref{lem:SOR0n1} and \ref{eqn:canonbasis}, we see that 
    \begin{align*}
        \omega_{0,m,1}(v,w_2,\ldots,w_m)&=vw_2\cdots w_m\\
        &=\sum_{a,b}vw_2\cdots w_{m-1}e_a\eta(\eta^{ab}e_b,w_m)\\
        &=\quad\vdots\\
        &=\sum_{\substack{a_1,\ldots,a_{m-2}\\ b_1,\ldots b_{m-2}}}vw_2e_{a_1}e_{a_2}\cdots e_{a_{m-2}}\eta(\eta^{a_1b_1}e_{b_1},w_3)\cdots \eta(\eta^{a_{m-2}b_{m-2}}e_{b_{m-2}},w_m)\\
        &=\sum_{\substack{a_1,\ldots,a_{m-2}, a_{m-1}\\ b_1,\ldots b_{m-2}, b_{m-1}}}\eta(vw_2e_{a_1}\cdots e_{a_{m-2}},e_{a_{m-1}})\eta(\eta^{a_1b_1}e_{b_1},w_3)\cdots \eta(\eta^{a_{m-2}b_{m-2}}e_{b_{m-2}},w_m)\\&\hspace{2.5cm}\tensor \eta^{a_{m-1}b_{m-1}}e_{b_{m-1}}\\
        &=\sum_{\substack{a_1,\ldots,a_{m-2}, a_{m-1}\\ b_1,\ldots b_{m-2}, b_{m-1}}}\eta(ve_{a_1}\cdots e_{a_{m-1}},w_2)\eta(\eta^{a_1b_1}e_{b_1},w_3)\cdots \eta(\eta^{a_{m-2}b_{m-2}}e_{b_{m-2}},w_m)\\&\hspace{2.5cm}\tensor \eta^{a_{m-1}b_{m-1}}e_{b_{m-1}}\\
        &=(\lambda^{\tensor m-1}\tensor\id)\circ\left(\sum_{\substack{a_1,\ldots,a_{m-2}, a_{m-1}\\ b_1,\ldots b_{m-2}, b_{m-1}}}ve_{a_1}\cdots e_{a_{m-1}}\tensor \eta^{a_1b_1}e_{b_1}\tensor\right.\\&\left.\hspace{3cm}\cdots\tensor \eta^{a_{m-1}b_{m-1}}e_{b_{m-1}}\right)(w_2,\ldots w_m).
    \end{align*}
    \noindent Now recognizing that $\omega_{0,m,1}=(\lambda^{\tensor m-1}\tensor \id) \circ \delta^{m-1}$, then $\omega_{0,1,m}=(r^{\tensor m-1}\tensor\id)\circ (\omega_{0,m,1})$ gives that $\omega_{0,1,m}(v)=\delta^{m-1}(v).$
\end{proof}
We can similarly extend this to all maps of genus 0, now with any number of inputs.
\begin{Lem}\label{lem:SOR0nm}
    For maps with genus 0 and $n,m\geq 1$, the value of a counital Almost TQFT is given by 
    \be\label{eqn:SOR0nm}
    \omega_{0,n,m}(v_1,\ldots, v_n)=\delta^{m-1}(v_1\cdots v_n).
    \ee
\end{Lem}
\begin{proof}
    The proof of this looks nearly identical to that of Lemma \ref{lem:SOR01m}, however everywhere we see a $v$, we replace it with the product $v_1\cdots v_n.$
\end{proof}
\noindent Now that we have all of the genus 0 maps, we can extend this to maps with positive genus using \ref{eqn:SOR1} to sew on $g$ copies of $\Sigma_{0,2,0}$ to a cobordism with a large enough number of output slots. We first examine this for genus 1 and then see how to extend to higher genus. 

\begin{Prop}\label{prop:SOR1nm}
    For maps with genus 1, the value of a counital Almost TQFT is given by 
    \be\label{eqn:SOR1nm}
     \omega_{1,n,m}(v_1,\ldots,v_n)=\begin{cases}
            \eps(\bige(v_1\cdots v_n)),& m=0,\\
            \bige(v_1\cdots v_n),& m=1,\\
            \delta^{m-1}(\bige(v_1\cdots v_n)),& m\geq 2.
        \end{cases}
    \ee
\end{Prop}
\begin{proof}
    For the $m=0$ case, we use the partial sewing axiom to note that
    \begin{align*}
        \omega_{1,n,0}(v_1,\ldots,v_n)&=(\omega_{0,2,0}\circ\omega_{0,n,2})(v_1,\ldots,v_n)\\
        &=\omega_{0,2,0}(\delta(v_1\cdots v_n))\\
        &=\eta\circ\left(\sum_{a,b}v_1\cdots v_ne_a\tensor \eta^{ab}e_b\right)\\
        &=\sum_{a,b}\eta(v_1\cdots v_ne_a,\eta^{ab}e_b)\\
        &=\sum_{a,b}\eps(v_1\cdots v_n \eta^{ab}e_ae_b)\\
        &=\eps(\bige(v_1\cdots v_n)).
    \end{align*}
    \noindent Now considering $m=1$, we have
    \begin{align*}
        \omega_{1,n,1}(v_1,\ldots,v_n)&=(\omega_{0,2,0}\circ\omega_{0,n,3})(v_1,\ldots,v_n)\\
        &=\omega_{0,2,0}(\delta^2(v_1\cdots v_n))\\
        &=(\eta\tensor\id)\circ\left(\sum_{a_1,a_2,b_1,b_2}v_1\cdots v_ne_{a_1}e_{a_2}\tensor \eta^{a_1b_1}e_{b_1}\tensor \eta^{a_2b_2}e_{b_2}\right)\\
        &=\sum_{a_1,a_2,b_1,b_2}\eta(v_1\cdots v_ne_{a_1}e_{a_2},\eta^{a_1b_1}e_{b_1}) \eta^{a_2b_2}e_{b_2}\\
        &=\sum_{a_1,a_2,b_1,b_2}\eta(v_1\cdots v_ne_{a_1}\eta^{a_1b_1}e_{b_1},e_{a_2}) \eta^{a_2b_2}e_{b_2}\\
        &=\sum_{a_2,b_2}\eta(\bige(v_1\cdots v_n,e_{a_2})\eta^{a_2b_2}e_{b_2}\\
        &=\bige(v_1\cdots v_n).
    \end{align*}

    \noindent Similarly, for $m\geq 2$, we have 
    \begin{align*}
        \omega_{1,n,m}(v_1,\ldots,v_n)&=(\omega_{0,2,0}\circ\omega_{0,n,m+2})(v_1,\ldots,v_n)\\
        &=\omega_{0,2,0}(\delta^{m+1}(v_1\cdots v_n))\\
        &=(\eta\tensor\id)\circ\left(\sum_{\substack{a_1,\ldots,a_{m+1}\\b_1,\ldots,b_{m+1}}}v_1\cdots v_ne_{a_1}e_{a_2}\cdots e_{a_{m+1}}\tensor \eta^{a_1b_1}e_{b_1}\tensor \cdots\tensor\eta^{a_{m+1}b_{m+1}}e_{b_{m+1}}\right)\\
        &=\sum_{\substack{a_1,\ldots,a_{m+1}\\b_1,\ldots,b_{m+1}}}\eta(v_1\cdots v_ne_{a_1}e_{a_2}\cdots e_{a_{m+1}},\eta^{a_1b_1}e_{b_1}) \eta^{a_2b_2}e_{b_2}\tensor \eta^{a_3b_3}e_{b_3}\tensor \eta^{a_{m+1}b_{m+1}}e_{b_{m+1}}
            \end{align*}
   
We therefore obtain
    \begin{align*}
             \omega_{1,n,m}(v_1,\ldots,v_n)   &=\sum_{\substack{a_1,\ldots,a_{m+1}\\b_1,\ldots,b_{m+1}}}\eta(v_1\cdots v_ne_{a_1}\cdots e_{a_{m+1}}\eta^{a_1b_1}e_{b_1},e_{a_2}) \eta^{a_2b_2}e_{b_2}\tensor \eta^{a_3b_3}e_{b_3}\tensor \eta^{a_{m+1}b_{m+1}}e_{b_{m+1}}\\
        &=\sum_{\substack{a_2,\ldots,a_{m+1}\\b_2,\ldots,b_{m+1}}}\eta(\bige(v_1\cdots v_nea_3\cdots e_{a_{m+1}},e_{a_2})\eta^{a_2b_2}e_{b_2}\tensor \eta^{a_3b_3}e_{b_3}\tensor \eta^{a_{m+1}b_{m+1}}e_{b_{m+1}}\\
        &=\sum_{\substack{a_3,\ldots,a_{m-1}\\b_3,\ldots,b_{m+1}}}\bige(v_1\cdots v_n)e_{a_3}\cdots e_{a_{m+1}}\eta^{a_3b_3}e_{b_3}\tensor \eta^{a_{m-1}b_{m+1}}e_{b_{m+1}}\\
        &=\delta^{m-1}(\bige(v_1\cdots v_n)).
    \end{align*}
\end{proof}

\noindent Following this, we can then induct on the genus to find all positive genus values of Almost TQFT.

\begin{Thm}[Theorem \ref{thm C}]\label{thm:SORgnm}
    The value of a counital Almost TQFT is given by
 
        \be\label{eqn:SORgnm}
     \omega_{g,n,m}(v_1,\ldots,v_n)=\begin{cases}
            
            \bige^g(v_1\cdots v_n),& m=1\\
            \eps(\omega_{g,n,1}(v_1,\ldots,v_n)),& m=0\\
            \delta^{m-1}(\omega_{g,n,1}(v_1,\ldots,v_n)),& m\geq 2
        \end{cases}.
    \ee
\end{Thm}

\begin{proof}[Proof of Theorem \ref{thm C}]

    Here we induct on $g$ and follow the same process of the proof of \ref{prop:SOR1nm}. Using the genus 0 values of \ref{lem:SOR0n1}, \ref{lem:SOR0n0}, and \ref{lem:SOR0nm} along with the genus 1 values of \ref{prop:SOR1nm} as base cases, we the assume that \ref{eqn:SORgnm} holds for genus $g-1$. We then follow the proof of \ref{prop:SOR1nm}, using the fact that \be
    \omega_{g,n,m}(v_1,\ldots,v_n)=(\omega_{0,2,0}\circ \omega_{g-1,n,m+2})(v_1,\ldots v_m),
    \ee
    \noindent using the basis expansions \ref{eqn:canonbasis} and \ref{eqn:EULbasis} to show that \ref{eqn:SORgnm} also holds for genus $g$.

\vskip 10pt

\noindent We can clearly see how this formula may arise from the topology of the associated cobordisms. Any cobordism with $n$ incoming boundary components, $m$ outgoing boundary components, and genus $g$ is topologically equivalent (and equivalent under the \cob relations) to a cobordism of the following form in three parts. The first part is a series of $n-1$ pairs of pants, sewn together in series with the appropriate number of cylinders. The second part is $2g$ pairs of pants sewn together in opposite directions, creating a hole with each pair. The final part is a series of $m-1$ pairs of pants in the reverse direction sewn together with the appropriate number of cylinders. Here we note that if $n=0$ or $m=0$, then instead of these pairs of pants, we simply replace this part of the cobordism with a cap in the appropriate direction.

\begin{figure}[ht]
    \centering
    \includegraphics[width=0.5\linewidth]{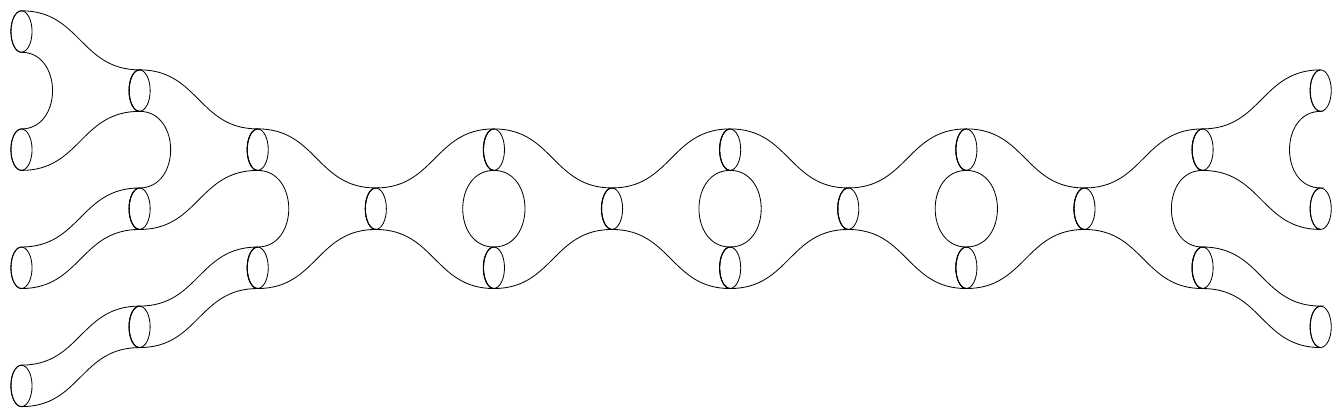}
    \caption{Normal form of a cobordism with $(g,n,m)=(3,4,3)$.}
    \label{fig:normalformcobordism}
\end{figure}
    
\end{proof}
\newpage
We can now complete the proof of the Corollary \ref{thm D}
\begin{proof}[Proof of Corollary \ref{thm D}] For $A$, a nearly Frobenius algebra, the classification of Almost TQFTs, Theorem \ref{thm C} and the classification of ribbon TQFTs, Theorem \ref{thm A} carried in \cite{DaDu} are the same. It implies that the list of axioms are equivalent.

\end{proof}

\nocite{*}
\setlength\bibitemsep{2\itemsep}
\printbibliography[title=REFERENCES]
\end{document}